\documentclass{article}

\usepackage{arxiv}

\usepackage[utf8]{inputenc} 
\usepackage[T1]{fontenc}    
\usepackage{hyperref}       
\usepackage{url}            
\usepackage{booktabs}       
\usepackage{amsfonts}       
\usepackage{nicefrac}       
\usepackage{microtype}      
\usepackage{lipsum}
\usepackage{graphicx}
\usepackage[most]{tcolorbox}
\usepackage{algorithm}
\usepackage{algpseudocode}

\usepackage{natbib}
\usepackage{subfig}
\usepackage{acro}
\usepackage{geometry}
\usepackage{comment}
\usepackage{xspace}
\usepackage{amsmath,bm}
\usepackage{longtable}
\usepackage{overpic}

\newtheorem{theorem}{Theorem}[section]
\newtheorem{proposition}{Proposition}[section]
\newenvironment{proof}{\paragraph{Proof:}}{\hfill$\square$}

\title{AN INTRODUCTION TO POD-GREEDY-GALERKIN REDUCED BASIS METHOD}

\author{
 Pierfrancesco Siena \\
  MathLab, Mathematics Area, \\
  SISSA International School for Advanced Studies, \\
    Via Bonomea, 265, 34136, Trieste, Italy,\\
  \texttt{psiena@sissa.it} \\
   \And
 Michele Girfoglio \\
  MathLab, Mathematics Area, \\
  SISSA International School for Advanced Studies, \\
    Via Bonomea, 265, 34136, Trieste, Italy,\\
  \texttt{mgifogl@sissa.it} \\
  \And
 Gianluigi Rozza \\
  MathLab, Mathematics Area, \\
  SISSA International School for Advanced Studies, \\
    Via Bonomea, 265, 34136, Trieste, Italy,\\
  \texttt{grozza@sissa.it} \\
}

\begin{document}
\newcommand{\openfoam}{Open\nolinebreak\hspace{-.2em}\nolinebreak\hspace{-.2em}FOAM\textsuperscript{\textregistered}\xspace}
\maketitle

\section*{ABBREVIATIONS}
{\renewcommand\arraystretch{0.8}
\begin{longtable}{p{2cm}p{11cm}p{10cm}}
\hspace{-45pt}	\textbf{FOM}  & Full order model\\
\hspace{-45pt}	\textbf{PDE}  & Partial differential equation \\
\hspace{-45pt}	\textbf{POD}  & Proper orthogonal decomposition\\
\hspace{-45pt}	\textbf{RB}  & Reduced basis \\
\hspace{-45pt}	\textbf{ROM}  & Reduced order method\\
\end{longtable}}

\section{INTRODUCTION AND MOTIVATION}\label{intro}
PDEs can be used to model many problems in several fields of application including, e.g.,  fluid mechanics, heat and mass transfer, and electromagnetism. 
Accurate discretization methods (e.g., finite element or finite volume methods, the so-called FOM) are widely used to numerically solve these problems. However, when many physical and/or geometrical parameters are involved, the computational cost required by FOMs becomes prohibitively expensive and this is not acceptable for real-time computations that are becoming more and more popular for rapid prototyping. Therefore, there is the need to introduce ROMs (also referred to as RB methods) able to provide, as the input parameters change, fast and reliable solutions at a reduced computational cost. 



The basic idea of ROM is related to the fact that often the parametric dependence of the problem 
at hand has an intrinsic dimension much lower than the number of degrees of freedom associated to the FOM. So the core of a ROM framework consists of computing a basis to be used to reconstruct adequately the solution of the problem. Different techniques have been explored during the last decades to generate the reduced space: see, e.g., \citealp{porsching1985estimation,ito1998reduced,lassila2013generalized, bui2003proper,christensen1999evaluation,gunzburger2007reduced, patera2007reduced,rozza2008reduced}. Of course the POD and the greedy algorithm are the most used. 
Typically, in a one dimensional parameter domain, a POD procedure is preferred, whilst for multi-dimensional spaces, strategies of greedy nature are favorite. This is due to the fact that the singular value decomposition of big snapshots matrices could lead to an high computational cost. 
Further details about advantages and disadvantages of greedy and POD strategies can be found in \cite{patera2007reduced,rozza2008reduced}. 

Preliminary studies showed a limited computational improvement provided by ROM, due to the lack of a full decoupling between ROM and FOM (\citealp{noor1981recent,porsching1987reduced,porsching1985estimation}). On the other hand, more recent works achieved a complete decoupling framework by means of an \emph{offline}-\emph{online} paradigm. The \emph{offline} stage is related to the collection of a database of several high-fidelity solutions  by solving the FOM for different values of physical and/or geometrical parameters. Then all the solutions are combined and compressed to extract a set of basis functions (computed by using the POD, the greedy algorithm or other techniques) that approximates the low-dimensional manifold on which the solution lies. On the other hand, in the \emph{online} stage the information obtained in the \emph{offline} stage is used to efficiently compute the solutions for new parameters instances. 
An high level of accuracy for the reduced system is then reached with the employment of a posterior error estimation technique (\citealp{ito1998reduced1,peterson1989reduced,balmes1996parametric,prud2002reliable,huynh2007successive}). In addition, the recent introduction of the empirical interpolation method allowed to reach an efficient \emph{offline}-\emph{online} decomposition also for complex problems involving non linear models (\citealp{barrault2004empirical,grepl2007efficient}).

Thanks to the significant research work carried out in the last
years, ROMs are able to provide a very general framework that has successfully been applied to a wide range of problems from different physical contexts, e.g. Navier-Stokes equations for fluid dynamics applications within both biomedical and industrial framework including multiphysics scenarios such as turbulence, multiphase flows and fluid-structure interaction (\citealp{ito1998reduced,gunzburger2012finite,ito1998reduced1,ito2001reduced,peterson1989reduced,girfoglio2021pod,girfoglio2021pressure,girfoglio2020non,prud2002reliable,cuong2005certified,rozza2005reduced,veroy2003posteriori,manzoni2012computational,quarteroni2007numerical,pintore2021efficient,hess1812spectral,hess2017spectral,deparis2009reduced, stabile2017pod,stabile2018finite,papapicco2021neural,ballarin2016pod,nonino2019overcoming,rozza2018advances,cuong2005certified,ballarin2017numerical,ballarin2016fast,Siena2022,zainib2020reduced,hijazi2020data, lovgren2006reduced0,lovgren2006reduced1,lovgren2006reduced2}), 
heat transfer and thermo-mechanical problems (\citealp{shah2021finite,guerin2019thermomechanical,hernandez2021model,benner2019comparison,grepl2005reduced,grepl2005posteriori,rozza2009reduced1}),  
Maxwell equations for the electromagnetism (\citealp{chen2009improved,chen2010certified,chen2012certified,jabbar2004fast}) 
and some toy models for the basic representation of complex physical phenomena, for instance the Burgers equation (\citealp{veroy2003posteriori,veroy2003reduced,deparis2009reduced}). 
Moreover, although most ROM works have been developed in a finite element environment, we highlight that also other discretization methods have been investigated, e.g. finite volume 
(\citealp{girfoglio2021pod,girfoglio2021pressure,girfoglio2020non, stabile2017pod, stabile2018finite, papapicco2021neural, hijazi2020data}) 
and spectral element (\citealp{lovgren2006reduced0,lovgren2006reduced1,lovgren2006reduced2, pintore2021efficient,hess1812spectral,hess2017spectral}) methods. 
However, since the aim of this chapter is to provide a brief overview of the main ingredients of ROMs, we limit to consider projection-based ROMs for the class of elliptic coercive PDEs in a finite element setting. For a deeper discussion about ROM including applications in different contexts, we refer to \cite{benner2020model}.

This chapter is structured as follows. A brief description about affine linear elliptic coercive problems is reported in  {{Section}} \ref{formulation}. The ROM methodology is explained in  {{Section}} \ref{rb}. In particular the POD and the greedy algorithms to build the reduced space are introduced. 
{{Section}} \ref{error} provides some notes about the posterior error bounds for the ROM approach. Finally, conclusions are provided in {{Section}} \ref{concl}.

\section{PARAMETRIZED DIFFERENTIAL EQUATIONS}\label{formulation}
Let us  consider the domain $\Omega \subset \mathbf{R}^d$ where $d=1,2$ or $3$, and $\partial \Omega$ is its boundary. Both vector fields ($d_v = d$) and scalar fields ($d_v = 1$) may be considered: $w:\Omega \mapsto \mathbf{R}^{d_v}$. The portions of $\partial \Omega$ where Dirichlet boundary conditions are imposed are identified by $\Gamma^D_i$, with $1\leq i \leq d_v$. 

Let $\mathbf{V}_i(\Omega)$ be the scalar space:
\begin{equation}
    \mathbf{V}_i(\Omega)=\{ v\in H^1(\Omega) : v_{\mid \Gamma^D_i}=0 \}, \qquad  1\leq i \leq d_v.
\end{equation}
We have 
$H^1_0(\Omega) \subset \mathbf{V}_i(\Omega) \subset H^1(\Omega)$ and for $ \Gamma^D_i = \partial \Omega $, $\mathbf{V}_i (\Omega) = H^1_0(\Omega)$. From now on, we refer to $\mathbf{V}_i(\Omega)$ as $\mathbf{V}_i$, for an easier exposition. 

Let us define the space $\mathbf{V}= \mathbf{V}_1 \times \dots \times \mathbf{V}_{d_v}$, whose general element is given by $w = (w_1, \dots, w_{d_v})$. The space $\mathbf{V}$ is equipped with an inner product denoted as $(w,v)_{\mathbf{V}}, \forall w,v \in \mathbf{V}$ and the induced norm is $\lVert w \rVert_{\mathbf{V}}=(w,w)_{\mathbf{V}}^{1/2}, \forall w \in \mathbf{V}$. We observe that $\mathbf{V}$ is an Hilber space. 

Let us introduce a closed parameter domain $\mathbf{P} \subset \mathbf{R}^p$. An element of $\mathbf{P}$ is given by $\mu = (\mu_1, \dots, \mu_p)$. So we define the parametric field variable for a parameter value,  $u(\mu)=(u_1 (\mu),\dots, u_{d_v}(\mu)):\mathbf{P}\mapsto\mathbf{V}$.

\subsection{Parametric weak formulation}\label{abstract-form}
The abstract formulation of a stationary problem reads:
\begin{tcolorbox}[colback=lightgray!5!white,colframe=lightgray!75!black]
  Given $\mu \in \mathbf{P} \subset \mathbf{R}^P$, find $u(\mu)\in \mathbf{V}$ such that
  \begin{equation}
      a(u(\mu),v;\mu) = f(v;\mu) \qquad \forall v \in \mathbf{V}
      \label{varitional-problem}
  \end{equation}
  and evaluate
  \begin{equation}
      s(\mu) = l(u(\mu);\mu).
      \label{varitional-problem1}
  \end{equation}
\end{tcolorbox}
The form $a:\mathbf{V}\times\mathbf{V}\times\mathbf{P}\mapsto\mathbf{R}$ is bilinear with respect to $\mathbf{V} \times \mathbf{V}$, $f:\mathbf{V}\times\mathbf{P}\mapsto\mathbf{R}$ and $l:\mathbf{V}\times\mathbf{P}\mapsto\mathbf{R}$ are linear with respect to $\mathbf{V}$ and $s:\mathbf{P}\mapsto\mathbf{R}$ is an output of interest of the model. 

Hereinafter we assume that the problems of interest are compliant, i.e.:
\begin{itemize}
    \item $f(\cdot;\mu) = l(\cdot;\mu) \quad \forall \mu \in \mathbf{P}$,
    \item $a(\cdot,\cdot;\mu)$ is symmetric $\forall \mu \in \mathbf{P}$.
\end{itemize}
The compliant hypothesis is verified for a wide range of case studies by allowing to significantly simplify the mathematical framework. 

\subsection{Well-posedness of the problem}\label{well-pose}
The well-posedness of the abstract problem formulation \ref{varitional-problem} could be established by using the Lax-Milgram theorem (\citealp{quarteroni2008numerical}) 
based on the following two assumptions:  

\begin{enumerate}
    \item The bilinear form $a(\cdot,\cdot;\mu)$ is coercive and continuous on $\mathbf{V} \times \mathbf{V}$. \\ The coercivity requires that $\forall \mu \in \mathbf{P}$, $\exists \alpha(\mu): 0<\alpha\leq \alpha(\mu) $ such that:
    \begin{equation}
        a(v,v;\mu)  \ge \alpha(\mu) \lVert v \rVert_{\mathbf{V}}^2, \quad \forall v\in \mathbf{V}.
    \end{equation}
    The coercivity constant $\alpha(\mu)$ is defined as:
    \begin{equation}
        \alpha(\mu)=\inf_{v \in \mathbf{V}}\frac{a(v,v;\mu)}{\lVert v \rVert_{\mathbf{V}}^2}, \quad \forall \mu \in \mathbf{P}.
        \label{coercivity}
    \end{equation}
    On the other hand, the continuity requires that $\forall \mu \in \mathbf{P}$, $\exists \gamma(\mu): \gamma(\mu)\leq\gamma < \infty $  such that:
    \begin{equation}
        a(w,v;\mu)\le\gamma(\mu)\lVert w \rVert_{\mathbf{V}}\lVert v \rVert_{\mathbf{V}}, \quad \forall w,v\in \mathbf{V},
    \end{equation}
    where the continuity constant $\gamma(\mu)$ is defined as:
    \begin{equation}
        \gamma(\mu) = \sup_{v \in \mathbf{V}}\sup_{w \in \mathbf{V}}\frac{a(w,v;\mu)}{\lVert w \rVert_{\mathbf{V}} \lVert v \rVert_{\mathbf{V}}}, \quad \forall \mu \in \mathbf{P}.
        \label{continuity}
    \end{equation}
    \item The linear form $f(\cdot;\mu)$ is continuous on $\mathbf{V}$, i.e. $\forall \mu \in \mathbf{P}$, $\exists \delta(\mu): \delta(\mu)\leq\delta < \infty $ such that:
    \begin{equation}
        f(v;\mu)\le \delta(\mu)\lVert v \rVert_{\mathbf{V}}, \quad \forall v\in \mathbf{V}.
    \end{equation}
\end{enumerate}
Note that the norm of $\mathbf{V}$ can be taken equal, or equivalent, to:
\begin{equation}
    \lVert v \rVert_{\mathbf{V}}=a(w,w;\bar\mu)^{1/2},\quad \forall w \in \mathbf{V},
\end{equation}
which is the norm induced by $a(\cdot,\cdot;\bar \mu)$ for a fixed parameter $\bar \mu \in \mathbf{P}$. 

\subsection{Discretization of the weak formulation}\label{discr-weak-form}
Let us consider a discrete and finite dimensional subset $\mathbf{V}_{\delta} \subset \mathbf{V}$ with $dim(\mathbf{V}_{\delta}) = N_\delta$. As an example, it can be built as a standard finite element method with a triangulation and defining suitable linear basis functions on each element (\citealp{bathe2007finite}). 

The discretized form of { Eqs.} \ref{varitional-problem} - \ref{varitional-problem1} is:
\begin{tcolorbox}[colback=lightgray!5!white,colframe=lightgray!75!black]
  Given $\mu \in \mathbf{P}$, find $u_{\delta}(\mu)\in \mathbf{V}_{\delta}$ such that
  \begin{equation}
      a(u_{\delta}(\mu),v_{\delta};\mu) = f(v_{\delta};\mu) \qquad \forall v_{\delta} \in \mathbf{V}_{\delta}
      \label{varitional-problem-discr}
  \end{equation}
  and evaluate
  \begin{equation}
      s_{\delta}(\mu) = l(u_{\delta}(\mu);\mu).
      \label{varitional-problem1-discr}
  \end{equation}
\end{tcolorbox}
{Eq.} \ref{varitional-problem-discr} is referred to as the high fidelity problem or the truth problem or the FOM. Commonly, high fidelity solutions are computed only when the number of parameters is small due to the usual high number of degrees of freedom $N_{\delta}$ despite the fact that the solution could be achieved with as high accuracy as desired. 

Thanks to the continuity and the coercivity of $a(\cdot,\cdot;\mu)$, the Galerkin orthogonality holds:
\begin{equation}
    a(u(\mu)-u_{\delta}(\mu),v_{\delta};\mu) = 0, \quad \forall v_{\delta} \in \mathbf{V}_{\delta}.
    \label{galerkin-ort}
\end{equation}
In addition, based on the triangle inequality, we obtain:
\begin{equation}
     \lVert u(\mu)-u_{\delta}(\mu) \rVert_{\mathbf{V}} \le  \lVert u(\mu)-v_{\delta} \rVert_{\mathbf{V}} + \lVert v_{\delta}-u_{\delta}(\mu) \rVert_{\mathbf{V}}, \quad \forall v_{\delta} \in \mathbf{V}_{\delta}.
     \label{triangle-in}
\end{equation}
Moreover, it holds that
\begin{equation}
\begin{split}
   \alpha(\mu)\lVert v_{\delta}- u_{\delta}(\mu) \rVert_{\mathbf{V}}^2 & \le  a( v_{\delta}- u_{\delta}(\mu), v_{\delta}- u_{\delta}(\mu);\mu) = a( v_{\delta}- u(\mu), v_{\delta}- u_{\delta}(\mu);\mu)
   \\
   & \le \gamma(\mu) \lVert v_{\delta}-u(\mu) \rVert_{\mathbf{V}} \lVert v_{\delta}-u_{\delta}(\mu) \rVert_{\mathbf{V}}, \quad \forall v_{\delta} \in \mathbf{V}_{\delta},
\end{split}
\end{equation}
and finally by applying the coercivity and continuity assumptions, and the Galerkin orthogonality, the Cea's lemma \citealp{monk2003finite} is recovered: 
\begin{equation}
    \lVert u(\mu)- u_{\delta}(\mu) \rVert_{\mathbf{V}} \le \bigg( 1 + \frac{\gamma(\mu)}{\alpha(\mu)} \bigg) \inf_{v_{\delta}\in \mathbf{V}_{\delta}} \lVert u(\mu)-v_{\delta} \rVert_{\mathbf{V}}, \quad \forall v_{\delta} \in \mathbf{V}_{\delta}.
\end{equation}
Thus the approximation error $\lVert u(\mu)- u_{\delta}(\mu) \rVert_{\mathbf{V}}$ is closely related to the best approximation error of $u(\mu)$ in the approximation space $\mathbf{V}_{\delta}$ through the constants $\alpha(\mu)$ and $\gamma(\mu)$.

Concerning the implementation of the truth solver, the stiffness matrix $A^{\mu}_{\delta}$ and the vector $f^{\mu}_{\delta}$ can be assembled as:
\begin{equation}
    (A^{\mu}_{\delta})_{i,j} =  a(\phi_j,\phi_i;\mu) \qquad \text{and} \qquad (f^{\mu}_{\delta})_i = f(\phi_i;\mu), \qquad 1 \le i,j \le N_{\delta},
\end{equation}
where $\{\phi_i \}_{i=1}^{N_{\delta}}$ is a basis of $\mathbf{V}_{\delta}$.
Therefore, the FOM problem in matrix form is: 
\begin{tcolorbox}[colback=lightgray!5!white,colframe=lightgray!75!black]
Given $\mu \in \mathbf{P}$, find $u^{\mu}_{\delta}\in \mathbf{R}^{N_{\delta}}$ such that:
\begin{equation}
A^{\mu}_{\delta}u^{\mu}_{\delta}=f^{\mu}_{\delta}.
\end{equation}
Then evaluate
\begin{equation}
    s_{\delta}(\mu) = {(u^{\mu}_{\delta})}^T f^{\mu}_{\delta}.
\end{equation}
The field approximation $u_{\delta}(\mu)$ is obtained by 
 $u_{\delta}(\mu)=\sum_{i=1}^{N_{\delta}}(u^{\mu}_{\delta})_i \phi_i$ where $(u^{\mu}_{\delta})_i$ denotes the $i$-th coefficient of the vector $u^{\mu}_{\delta}$.
\end{tcolorbox}
\section{REDUCED ORDER MODEL}\label{rb}
As disclosed in {Sec. }\ref{intro}, ROM is developed to deal with repeated model evaluation over a wide range of parameters values by cutting down the large computational cost due to solving the FOM many times. In particular, the aim is to extract from a collection of solutions of the parameterized
problem related to determined values of the parameters (named solution manifold)  
a small number of basis functions to be used in order to obtain an accurate reconstruction of the solution for a new instance of the parameters at a reduced computational cost. Such a goal is conduced by introducing an \emph{offline}-\emph{online} paradigm. The \emph{offline} stage requires the resolution of $N$ truth problems, therefore its computational cost is elevated due to the high value of $N_{\delta}$. Then a reduced basis of size $N$ is identified. On the other hand, during the \emph{online} stage, a Galerkin projection onto the space spanned by the reduced basis is carried out, enabling to investigate the parameter space at a considerably reduced cost. 

\subsection{The reduced approximation}
Let us define the solution manifold of {Eq.} \ref{varitional-problem} the set of all the exact solutions $u(\mu)\in\mathbf{V}$ varying the parameter $\mu$:
\begin{equation}
    \mathcal{M} = \{u(\mu) : \mu \in \mathbf{P} \} \subset \mathbf{V}.
\end{equation}
At a discrete level, we can define similarly the solution manifold of {Eq.} \ref{varitional-problem-discr}:
\begin{equation}
    \mathcal{M}_{\delta}=\{ u_{\delta}(\mu) : \mu \in \mathbf{P}\} \subset \mathbf{V}_\delta.
\end{equation}
Of course, the cost to find the truth solutions can be very large because it depends on $N_{\delta}$. 
ROM aims to find a small number of basis functions whose linear combination represents accurately the numerical solution $u_{\delta}(\mu)$. Let $\{ \xi \}_{i=1}^{N}$ be the $N$-dimensional set of the reduced basis (the main techniques to compute it are explained in {Section \ref{rb-gen}}). Then the space generated is:
\begin{equation}
    \mathbf{V}_{rb}=span\{ \xi_1,\dots,\xi_N \} \subset \mathbf{V}_{\delta},
\end{equation}
where it is assumed $N \ll N_{\delta}$. 
Therefore, the reduced form of the {Eqs.} \ref{varitional-problem-discr}-\ref{varitional-problem1-discr} is:
\begin{tcolorbox}[colback=lightgray!5!white,colframe=lightgray!75!black]
  Given $\mu \in \mathbf{P}$, find $u_{rb}(\mu)\in \mathbf{V}_{rb}$ such that
  \begin{equation}
      a(u_{rb}(\mu),v_{rb};\mu) = f(v_{rb};\mu) \qquad \forall v_{rb} \in \mathbf{V}_{rb}
      \label{varitional-problem-rb}
  \end{equation}
  and evaluate
  \begin{equation}
      s_{rb}(\mu) = f(u_{rb}(\mu);\mu).
      \label{varitional-problem1-rb}
  \end{equation}
  The reduced solution can be expressed as $    u_{rb}(\mu)=\sum_{i=1}^{N}(u_{rb}^{\mu})_i \xi_i$ where $ (u_{rb}^{\mu})_i$ are the coefficients of the reduced basis approximation.
\end{tcolorbox}

An essential analysis concerns the accuracy of ROM. By using the triangle inequality, it is possible to estimate the ROM error by controlling the distance between the truth solution and the reduced solution:
\begin{equation}
    \lVert u(\mu)- u_{rb}(\mu) \rVert_{\mathbf{V}} \le \lVert  u(\mu) - u_{\delta}(\mu) \rVert_{\mathbf{V}} + \lVert u_{\delta}(\mu) - u_{rb}(\mu)  \rVert_{\mathbf{V}},
\end{equation}
under the assumption  that there exists a numerical method which allows to have $\lVert  u(\mu) - u_{\delta}(\mu) \rVert_{\mathbf{V}}$ small enough (see {Section} \ref{discr-weak-form}). \\ Let us introduce the Kolmogorov $N$-width which measures the distance between $\mathcal{M}_{\delta}$ and $\mathbf{V}_{rb}$:
\begin{equation}
    d_N(\mathcal{M}_{\delta}) = \inf_{\mathbf{V}_{rb}} \mathcal{E}(\mathcal{M}_{\delta},\mathbf{V}_{rb}) = \inf_{\mathbf{V}_{rb}} \sup_{u_{\delta}\in\mathcal{M}_{\delta}}\inf_{v_{rb}\in\mathbf{V}_{rb}} \lVert u_{\delta}-v_{rb} \rVert_{\mathbf{V}}.
\end{equation}
We observe that when $d_N(\mathcal{M}_{\delta})$ decreases quickly as $N$ increases, a low number of basis functions could approximate properly the solution manifold $\mathcal{M}_{\delta}$. \\ In addition, the Cea's lemma applied to $\mathbf{V}_{rb}$ states:
\begin{equation}
    \lVert u(\mu)- u_{rb}(\mu) \rVert_{\mathbf{V}} \le \bigg( 1 + \frac{\gamma(\mu)}{\alpha(\mu)} \bigg) \inf_{v_{rb}\in \mathbf{V}_{rb}} \lVert u(\mu)-v_{rb} \rVert_{\mathbf{V}}, \quad \forall v_{rb} \in \mathbf{V}_{rb},
    \label{cea}
\end{equation}
where the coercivity and the continuity constants occur. Thus the ROM accuracy depends on the problem at hand. 
\\ The reduced solution matrix $A^{\mu}_{rb}$ and the vector $f^{\mu}_{rb}$ are:
\begin{equation}
    (A^{\mu}_{rb})_{m,n} =  a(\xi_n,\xi_m;\mu) \qquad \text{and} \qquad (f^{\mu}_{rb})_m = f(\xi_m;\mu), \qquad 1 \le m,n \le N.
\end{equation}
So the ROM in matrix form is: 
\begin{tcolorbox}[colback=lightgray!5!white,colframe=lightgray!75!black]
Given $\mu \in \mathbf{P}$, find $u^{\mu}_{rb}\in \mathbf{R}^{N}$ such that:
\begin{equation}
A^{\mu}_{rb}u^{\mu}_{rb}=f^{\mu}_{rb}
\end{equation}
and evaluate
\begin{equation}
    s_{rb}(\mu) = {(u^{\mu}_{rb})}^T f^{\mu}_{rb}.
\end{equation}

\end{tcolorbox}

\subsection{Reduced basis generation}\label{rb-gen}
In this section we analyse the two main strategies to generate the reduced basis space: the POD and the greedy algorithm. We highlight that, in literature, many other techniques can be found, also involving machine learning approaches: see, e.g., \cite{fu2021data,lee2020model,gonzalez2018deep,kashima2016nonlinear}. However the description of these other methods is out of the scope of this chapter. \\ Let us introduce a discrete and finite-dimensional set of parameters $\mathbf{P}_{h}$. It can be obtained from an equispaced or a random sampling of $\mathbf{P}$. Therefore, the following solution manifold of dimension $M=|\mathbf{P}_h|$ can be considered:
\begin{equation}
    \mathcal{M}_{\delta}(\mathbf{P}_h)=\{ u_{\delta}(\mu) : \mu \in \mathbf{P}_h  \}.
\end{equation}
By definition, it holds:
\begin{equation}
    \mathcal{M}_{\delta}(\mathbf{P}_h)\subset\mathcal{M}_{\delta} \qquad \text{and} \qquad \mathbf{P}_h \subset \mathbf{P}.
\end{equation}
Consequently, if $\mathbf{P}_h$ is sufficiently rich, $\mathcal{M}_{\delta}(\mathbf{P}_h)$ is able to represent properly $\mathcal{M}_{\delta}$.
\subsubsection{Proper orthogonal decomposition}
POD is a technique whose aim consists of squeezing data. 
Once the parameter space is discretized and the high fidelity solutions for each element of $\mathbf{P}_h$ are computed, POD extracts and retains only the essential information with a resulting compact form of the problem. \\ The POD-space of dimension $N$ is the solution of the following minimization problem:
\begin{equation}
    \min_{\mathbf{V}_{rb}:|\mathbf{V}_{rb}|=N} \Bigg( \int_{\mu \in \mathbf{P}} \inf_{v_{rb}\in \mathbf{V}_{rb}} \lVert u_{\delta}(\mu)-v_{rb} \rVert_{\mathbf{V}}^2 d\mu \Bigg)^{1/2},
\end{equation}
whose discrete version corresponds to:
\begin{equation}
     \min_{\mathbf{V}_{rb}:|\mathbf{V}_{rb}|=N} \Bigg( \frac{1}{M}\sum_{\mu \in \mathbf{P}_h} \inf_{v_{rb}\in \mathbf{V}_{rb}} \lVert u_{\delta}(\mu)-v_{rb} \rVert_{\mathbf{V}}^2  \Bigg)^{1/2}.
     \label{opt-crit}
\end{equation}
Let us sort the elements of $\mathbf{P}_h$
\begin{equation}
    \{\mu_1,\dots,\mu_M\},
\end{equation}
and denote with
\begin{equation}
    \{ \psi_1, \dots, \psi_M \}
\end{equation}
the elements of $\mathcal{M}_{\delta}(\mathbf{P}_h)$, $\psi_m=u_{\delta}(\mu_m)$ for $m=1,\dots,M$ (the so-called snapshots). \\ If $\mathbf{V}_{\mathcal{M}}=span\{ u_{\delta}(\mu) : \mu \in \mathbf{P}_h \}$, we can introduce the symmetric and linear operator $\mathcal{C}:\mathbf{V}_{\mathcal{M}}\mapsto\mathbf{V}_{\mathcal{M}}$ defined as:
\begin{equation}
    \mathcal{C}(v_{\delta})=\frac{1}{M}\sum_{m=1}^M (v_{\delta},\psi_m)_{\mathbf{V}} \psi_m \quad v_{\delta} \in \mathbf{V}_{\mathcal{M}},
    \label{linear-op}
\end{equation}
whose eigenvalues and normalized eigenvectors are denoted as $(\lambda_i,\xi_i)\in \mathbf{R}\times \mathbf{\mathbf{V}_{\mathcal{M}}}$ and satisfy:
\begin{equation}
    (\mathcal{C}(\xi_i),\psi_m)_{\mathbf{V}}=\lambda_i(\xi_i,\psi_m)_{\mathbf{V}}, \quad 1 \le m \le M.
    \label{eig-p}
\end{equation}
Here we assume that the eigenvalues are sorted in descending order, $\lambda_1 \ge \lambda_2 \ge \dots \ge \lambda_M$. The basis functions for the space $\mathbf{V}_{\mathcal{M}}$  are given by the eigenfunctions  $\{\xi_1,\dots,\xi_M\}$. \\ The first $N \ll M$ eigenfunctions $\{\xi_1,\dots,\xi_N\}$ generate the $N$-dimensional reduced space $\mathbf{V}_{POD} = span \{\xi_1,\dots,\xi_N \}$ which fulfills {Eq.} \ref{opt-crit}. 
By projecting the elements of $\mathcal{M}_{\delta}(\mathbf{P}_h)$ onto $\mathbf{V}_{POD}$, it can be proven that the error is related to the neglected eigenvalues:
\begin{equation}
    \frac{1}{M}\sum_{m=1}^M\lVert \psi_m - P_N[\psi_m] \rVert_{\mathbf{V}}^2 = \sum_{m=N+1}^M \lambda_m,
\end{equation}
where $P_N[\psi_m] = \sum_{i=1}^N(\psi_m,\xi_i)_{\mathbf{V}}\xi_i$ is the projection of $\psi_m$ onto $\mathbf{V}_{POD}$.
Due to the orthonormality of the eigenvectors in $\lVert \cdot \rVert_{\ell^2(\mathbf{R}^M)}$, we can derive the orthogonality property for the eigenfunctions:
\begin{equation}
    (\xi_m,\xi_q)_{\mathbf{V}}=M\lambda_i\delta_{mq}, \quad 1\le m,q \le M,
\end{equation}
where $\delta_{mq}$ is the Kronecker delta.

In matrix form, the correlation matrix $C\in\mathbf{R}^{M\times M}$ corresponding to the linear operator \ref{linear-op} is:
\begin{equation}
    C_{mq}=\frac{1}{M}(\psi_m,\psi_q)_{\mathbf{V}}, \quad 1 \le m,q \le M.
\end{equation}
The eigenvalues problem equivalent to {Eq.} \ref{eig-p} is:
\begin{equation}
    C v_i = \lambda_i v_i, \quad 1 \le i \le N.
\end{equation}
Then, the orthogonal basis functions are given by:
\begin{equation}
    \xi_i = \frac{1}{\sqrt{M}}\sum_{m=1}^M (v_i)_m \psi_m, \quad 1 \le i \le N,
\end{equation}
where$(v_i)_m$ denotes the $m$-th coefficient of the eigenvector $v_i \in \mathbf{R}^M$ .

The computational cost to perform the POD can be very high. In fact, it is not possible to know the number of high fidelity solutions needed to ensure a good solution a priori and it depends on the problem at hand. Therefore, a lot of FOMs, $M\gg N$, have often to be solved, leading to an expensive offline phase. In addition, when $M$ and $N_{\delta}$ are large, also the cost to extrapolate the eigenfunctions, scaling as $\mathcal{O}(NN_{\delta}^2)$, increases. \\



\subsubsection{Greedy algorithm}
The greedy algorithm, schematized in {Algorithm} \ref{alg:the_alg}, is an iterative strategy where at each iteration
one new basis function is added and the overall precision of the basis set is improved. It requires
one truth solution to be computed per iteration and a total of $N$ truth solutions to generate the
$N$-dimensional reduced basis space. \\ Let us suppose that an estimation $\eta(\mu)$ of the error due to the replacement of $\mathbf{V}_{\delta}$ with $\mathbf{V}_{rb}$ is available (see {Section} \ref{error} for further details), i.e.
\begin{equation}
    \lVert u_{\delta}(\mu)-u_{rb}(\mu) \rVert_{\mu} \le \eta(\mu), \quad \forall \mu \in \mathbf{P}.
\end{equation}
At the $n$-th step of the iterative process, a new parameter is selected:
\begin{equation}
    \mu_{n+1}=\underset{\mu \in \mathbf{P}}{\operatorname{argmax}} \big[  \eta(\mu) \big],
\end{equation}
and the corresponding full order solution $u_{\delta}(\mu_{n+1})$ is computed. Then it is included in the reduced basis space $\mathbf{V}_{rb}=span \{u_{\delta}(\mu_1),\dots,u_{\delta}(\mu_{n+1}) \} $. In other words, at each iteration, if the dimension of $\mathbf{V}_{rb}$ is $n$, the $n+1$ basis function maximizes the estimation of the model order reduction error over $\mathbf{P}$. This is repeated until the maximal estimated error is under a fixed tolerance. So while the POD procedure find a basis which is optimal for the $L^2$-norm, the greedy approach is based on the maximum norm over $\mathbf{P}$. \\From an operational point of view, as for the POD approach, we introduce the discrete parameter space $\mathbf{P}_h$. Since the greedy approach needs only the evaluation of the error estimator and not the resolution of {Eq.} \ref{varitional-problem-discr} for each point in $\mathbf{P}_h$, 
the parameter space can be denser
than the one used with the POD procedure, simply because the computational cost is considerably smaller. So the main advantages of this approach are two: the resolution of a big eigenvalues problem is avoided and only $N$ truth
solutions are computed. 
\\ As a general rule, 
if $F=\{ f(\mu):\Omega \mapsto \mathbf{R} : \mu \in \mathbf{P} \}$ is a set of parametrized functions, the $n+1$-th basis function is chosen as
\begin{equation}
    f_{n+1} = \underset{\mu \in \mathbf{P}}{\operatorname{argmax}}  \lVert f(\mu)- P_n f(\mu) \rVert_{\mathbf{V}} ,
\end{equation}
where $P_nf$ is the orthogonal projection onto $F_n=span\{f_1,\dots,f_n \}$. Some results about the convergence of the greedy algorithm are shown in the following (\citealp{binev2011convergence,devore2013greedy}). 
\begin{theorem}
Let us assume that $F$ has an exponentially small Kolmogorov $N$-width, $d_N(F) \le ce^{-aN}$, with $a>log(2)$. Then, $\exists \beta >0 $ such that the set $F_N=span\{f_1,\dots,f_N \}$ resulting from the greedy algorithm is exponentially accurate in the sense that: $$ \lVert f - P_Nf \rVert_{\mathbf{V}} \le Ce^{-\beta N}$$
\end{theorem}
Hence, if a problem allows an efficient and compact reduced basis, the greedy algorithm will provide an exponentially convergent approximation to it.
\begin{theorem}
Let us assume that all the solutions $\mathcal{M}$ have an exponentially small Kolmogorov $N$-width, $d_N(\mathcal{M}) \le ce^{-aN}$, with $a>log\big( 1+ \sqrt{\frac{\gamma}{\alpha}} \big)$ (where $\gamma$ and $\alpha$ are the continuity and coercivity constants introduced in {Section} \ref{well-pose}). Then, the reduced basis approximation converges exponentially fast in the sense that $\exists \beta >0 $ such that:
$$\forall \mu \in \mathbf{P} : \lVert u_{\delta}(\mu)-u_{rb}(\mu) \rVert_{\mathbf{V}}\le Ce^{-\beta N}$$
\end{theorem}
Therefore the reduced basis approximation $u_{rb}(\mu)$ goes exponentially to the truth
solution $u_{\delta}(\mu)$. Note that it is not sufficient in order to conclude that the reduced method works properly, in fact it is necessary to suppose that truth solution is a good approximation of the exact solution $u(\mu)$.

It is important to observe that, using a greedy approach, the snapshots for each parameter in $\mathbf{P}_h$, $ u_{\delta}(\mu_1),\dots, u_{\delta}(\mu_N) $, may be linearly dependent, by leading to a large condition number of
the associated solution matrix. Therefore, in order to overcome such an issue, it is recommended to orthonormalize $ u_{\delta}(\mu_1),\dots, u_{\delta}(\mu_N) $ in order to find the reduced basis $\xi_1,\dots,\xi_N$.


\begin{algorithm}
\caption{The greedy argorithm}\label{pod-alg}
\text{\textbf{Input:} $\mu_1$, tolerance}\\
\text{\textbf{Output:} $\mathbf{V}_{rb}$}
\begin{algorithmic}[1]
\State Compute $u_{\delta}(\mu_n)$
\State $\mathbf{V}_{rb}=span\{u_{\delta}(\mu_1),\dots,u_{\delta}(\mu_n) \}$
 \For{$\mu \in \mathbf{P}_h$}
 \State compute $u_{rb}(\mu)$
  \State evaluate $\eta(\mu)$
 \EndFor
\State $ \mu_{n+1}=\underset{\mu \in \mathbf{P}}{\operatorname{argmax}} \big[  \eta(\mu) \big],$

\If{$\eta(\mu_{n+1}) >$ tolerance}
    \State $n=n+1$

\Else{}
    \State break.
\EndIf

\end{algorithmic}
\label{alg:the_alg}
\end{algorithm}

\subsection{Notes on the affine decomposition}
The computation of the reduced solution $u_{rb}(\mu)$ for each new parameter $\mu \in \mathbf{P}$ is the main task of the online phase. The associated computational cost should be independent on the complexity of the truth problem and depend only on the reduced basis space $N$. 
Therefore, in order to ensure the efficiency of the reduced order framework, the affine decomposition strategy is considered. The following forms are introduced: 
\begin{equation}
    a_q:\mathbf{V}\times\mathbf{V}\mapsto\mathbf{R}, \qquad f_q:\mathbf{V}\mapsto\mathbf{R}, \qquad l_q:\mathbf{V}\mapsto\mathbf{R},
\end{equation}
which are independent on the parameter $\mu \in \mathbf{P}$ so that
\begin{equation}
    a(u,v;\mu)=\sum_{q=1}^{Q_a} \theta^q_a(\mu)a_q(w,v),
\end{equation}
\begin{equation}
    f(v;\mu) = \sum_{q=1}^{Q_f}\theta^q_f(\mu) f_q(v),
\end{equation}
\begin{equation}
    l(v;\mu)=\sum_{q=1}^{Q_l} \theta_l^q(\mu)l_q(v).
\end{equation}
The quantities $\theta^q_a$,$\theta^q_f$,$\theta^q_l$ are the scalar coefficients depending only on the parameter values $\mu \in \mathbf{P}$. \\ As result of this decomposition, if $A^q_{rb}$ is the matrix associated to $a_q(\cdot,\cdot)$ and $f^q_{rb}$ and $l^q_{rb}$ are the vectors associated to $f_q(\cdot)$ and $l_q(\cdot)$, respectively, they can be precomputed during the offline phase (because independent on the parameter $\mu$). Then, during the online phase, one obtains:
\begin{equation}
    A^{\mu}_{rb}=\sum_{q=1}^{Q_a}\theta^q_a(\mu)A^q_{rb}, \qquad   f^{\mu}_{rb}=\sum_{q=1}^{Q_f}\theta^q_f(\mu)f^q_{rb},\qquad   l^{\mu}_{rb}=\sum_{q=1}^{Q_l}\theta^q_l(\mu)l^q_{rb}.
\end{equation}
If the problem
does not allow an affine representation, it can be approximated using other techniques, as the empirical interpolation method (see, e.g.,  \citealp{hesthaven2016certified}).

\section{ERROR ANALYSIS}\label{error}
An accurate posterior error estimation is desirable
to have a robust and powerful reduced order model. Reliable approximation of the error is a significant aspect during both the offline and online phases. In fact, ROMs are problem dependent and can not be directly related to specific spatial
or temporal scales, so problem intuition is of limited value and may even be wrong. In addition, sometimes they are used for real-time predictions when no time is available for an offline check. \\ When the parameter space is very large, some regions of $\mathbf{P}$ can be unexplored due to the sampling strategies chosen to obtain $\mathbf{P}_h$. 
Consequently, only the error in $\mathbf{P}_h$ and not also in $\mathbf{P}$ can be bounded. However, it is possible to control the output error during the online phase with several posteriori error estimations, for every $\mu \in \mathbf{P}$, as we are going to show in this section. Some properties have to be satisfied: the error bounds has to be rigorous and valid for all $N$ and for every $\mu \in \mathbf{P}$; furthermore, the bounds must be reasonably sharp since overly conservative errors
will yield inefficient approximations; finally, the bounds have to be computable at low cost,
independent on $N_\delta$.

\subsection{Error representation}
Let us introduce the discrete version of the coercivity and the continuity constants ({Eqs.} \ref{coercivity} and \ref{continuity}): 
    \begin{equation}
        \alpha_{\delta} (\mu)=\inf_{v_{\delta} \in \mathbf{V}_{\delta}}\frac{a(v_{\delta},v_{\delta};\mu)}{\lVert v_{\delta} \rVert_{\mathbf{V}}^2}, \qquad 
        \gamma_{\delta}(\mu) = \sup_{v_{\delta} \in \mathbf{V}_{\delta}}\sup_{w_{\delta} \in \mathbf{V}_{\delta}}\frac{a(w_{\delta},v_{\delta};\mu)}{\lVert w_{\delta} \rVert_{\mathbf{V}} \lVert v_{\delta} \rVert_{\mathbf{V}}}.
        \label{coerc-discr}
    \end{equation} 
By considering that $\mathbf{V}_{\delta}\subset\mathbf{V}$, one obtains: 
\begin{equation}
    \alpha(\mu)\le\alpha_{\delta}(\mu), \qquad \gamma_{\delta}(\mu)\le\gamma(\mu).
\end{equation}
If the error is $e(\mu)=u_{\delta}(\mu)-u_{rb}(\mu)\in \mathbf{V}_{\delta}$ from the bilinearity of $a(\cdot,\cdot;\mu)$ it holds: 
\begin{equation}
    a(e(\mu),v_{\delta};\mu) = r(v_{\delta};\mu), \quad \forall v_{\delta} \in \mathbf{V}_{\delta},
    \label{error-eq}
\end{equation}
where the residual $r(\cdot,\mu)$ in the dual space $\mathbf{V}_{\delta}^*$ is:
\begin{equation}
    r(v_{\delta};\mu)=f(v_{\delta};\mu)-a(u_{rb}(\mu),v_{\delta};\mu),\quad \forall v_{\delta} \in \mathbf{V}_{\delta}.
\end{equation}
{Eq.} \ref{error-eq} can be written as:
\begin{equation}
    a(e(\mu),v_{\delta};\mu) = (\hat{r}_{\delta}(\mu),v_{\delta})_{\mathbf{V}}, \quad \forall v_{\delta} \in \mathbf{V}_{\delta},
    \label{error-eq1}
\end{equation}
where $\hat{r}_{\delta}(\mu) \in \mathbf{V}_{\delta}$ is the Riesz representation of $r(\cdot,\mu)\in \mathbf{V}_{\delta}^*$:
\begin{equation}
    (\hat{r}_{\delta}(\mu),v_{\delta})_{\mathbf{V}}=r(v_{\delta};\mu),\quad \forall v_{\delta} \in \mathbf{V}_{\delta}.
\end{equation}
Therefore it holds:
\begin{equation}
    \lVert \hat{r}_{\delta}(\mu) \rVert_{\mathbf{V}} = \lVert r(\cdot;\mu) \rVert_{\mathbf{V}_{\delta}^*} =  \sup_{v_{\delta}\in\mathbf{V}_{\delta}}\frac{r(v_{\delta};\mu)}{\lVert v_{\delta} \rVert_{\mathbf{V}}},\quad \forall v_{\delta} \in \mathbf{V}_{\delta}.
\end{equation}
Due to the hypothesis of compliant problem (see {Sec.} \ref{abstract-form}), it can be introduced the { Prop.} \ref{prop-err} about some error relations, useful for future proofs.
\begin{proposition}
\label{prop-err}
If the problem is compliant (see {Sec.} \ref{abstract-form}), then $\forall \mu \in \mathbf{P}$:$$
    s_{\delta}(\mu)-s_{rb}(\mu)=\lVert u_{\delta}(\mu)-u_{rb}(\mu) \rVert_{\mu}^2,\quad 
    \Longrightarrow \quad s_{\delta}(\mu) \ge s_{rb}(\mu).$$
\end{proposition}
\begin{proof}
Let us consider $u_{\delta}(\mu) \in \mathbf{V}_{\delta}$ and $u_{rb}(\mu)\in \mathbf{V}_{rb}$ and recall the Galerkin orthogonality:
\begin{equation}
        a(u_{\delta}(\mu)-u_{rb}(\mu),v_{\delta};\mu) = 0, \quad \forall v_{rb} \in \mathbf{V}_{rb}.
\end{equation}
For the linearity of $f(\cdot;\mu)$, the hypothesis of compliant problem and the Galerkin orthogonality, we get:
\begin{equation}
    s_{\delta}(\mu)-s_{rb}(\mu)=f(e(\mu);\mu)=a(u_{\delta}(\mu),e(\mu);\mu)=a(e(\mu),e(\mu);\mu)=\lVert e(\mu) \rVert_{\mu}^2.
\end{equation}
\end{proof}

\subsection{Error bounds}
Let us assume to have a lower bound $\alpha_{LB}(\mu)$ for the discrete coercivity constant $\alpha_{\delta}(\mu)$ ({Eq.} \ref{coerc-discr}), indipendent on $N_{\delta}$. \\ Let us introduce the following error estimators:
\begin{equation}
    \eta_{en}(\mu)=\frac{\lVert\hat{r}_{\delta}(\mu)\rVert_{\mathbf{V}}}{\alpha_{LB}^{1/2}(\mu)}, \quad  \text{for the energy norm},
\end{equation}

\begin{equation}
    \eta_{s}(\mu)=\frac{\lVert\hat{r}_{\delta}(\mu)\rVert_{\mathbf{V}}^2}{\alpha_{LB}(\mu)}= \eta_{en}^2(\mu), \quad  \text{for the output},
\end{equation}

\begin{equation}
    \eta_{s,rel}(\mu)=\frac{\lVert\hat{r}_{\delta}(\mu)\rVert_{\mathbf{V}}^2}{\alpha_{LB}(\mu)s_{rb}(\mu)}= \frac{\eta_{s}(\mu)}{s_{rb}(\mu)}, \quad  \text{for the relative  output}.
\end{equation}
It holds the {Prop.} \ref{err-estimators} for the error of the reduced solution.
\begin{proposition}
\label{err-estimators}
For all $\mu \in \mathbf{P}$:
\begin{equation}
    \lVert u_{\delta}(\mu)-u_{rb}(\mu) \rVert_{\mu}\le \eta_{en}(\mu),
    \label{prima}
\end{equation}
\begin{equation}
     s_{\delta}(\mu)-s_{rb}(\mu) \le \eta_{s}(\mu),
     \label{seconda}
\end{equation}
\begin{equation}
     \frac{s_{\delta}(\mu)-s_{rb}(\mu)}{s_{\delta}(\mu)} \le \eta_{s,rel}(\mu).
     \label{terza}
\end{equation}
\end{proposition}

\begin{proof}
Using {Eq.} \ref{error-eq1} with $v_{\delta} = e(\mu)$ and the Cauchy-Schwarz inequality, it holds:
\begin{equation}
    \lVert e(\mu) \rVert_{\mu}^2 = a(e(\mu),e(\mu);\mu) \le \lVert \hat{r}_{\delta}(\mu) \rVert_{\mathbf{V}} \lVert e(\mu) \rVert_{\mathbf{V}}. \label{x}
\end{equation}
In addition, for the coercivity of $a(\cdot,\cdot;\mu)$ and by considering that $\alpha_{LB}(\mu)\le\alpha_{\delta}(\mu)$ by definition, one obtains:
\begin{equation}
    \alpha_{LB}(\mu)\lVert e(\mu) \rVert_{\mathbf{V}}^2 \le a(e(\mu),e(\mu);\mu) =\lVert e(\mu) \rVert_{\mu}^2. \label{xx}
\end{equation}
Combining thus {Eqs.} \ref{x} and \ref{xx} with  { Prop.} \ref{prop-err} yields 
{ Eqs.} \ref{prima}, \ref{seconda} and \ref{terza}. 
\end{proof}
Let us define the following effectivity indexes:
\begin{equation}
    eff_{en}(\mu) = \frac{\eta_{en}(\mu)}{\lVert u_{\delta}(\mu)-u_{rb}(\mu) \rVert_{\mu}},
\end{equation}
\begin{equation}
    eff_{s}(\mu) = \frac{\eta_{s}(\mu)}{ s_{\delta}(\mu)-s_{rb}(\mu) },
\end{equation}
\begin{equation}
    eff_{s,rel}(\mu) = \frac{\eta_{s,rel}(\mu)s(\mu)}{ s_{\delta}(\mu)-s_{rb}(\mu)}.
\end{equation}
These quantities are $\geq 1$ as ensured by { Prop.} \ref{err-estimators}, and the quality of the estimators increases as the effectivity indexes are closer to one. If the problem is coercive and compliant, they are bounded as specified in the next proposition.
\begin{proposition}
For all $\mu \in \mathbf{P}$, the effectivity indexes are controlled by an upper bound as follows: 
\begin{equation}
    eff_{en}(\mu)\le \sqrt{\gamma_{\delta}(\mu)/\alpha_{LB}(\mu)},
    \label{eff_en}
\end{equation}
\begin{equation}
    eff_{s}(\mu)\le \gamma_{\delta}(\mu)/\alpha_{LB}(\mu),
    \label{eff_s}
\end{equation}
\begin{equation}
    eff_{s,rel}(\mu)\le (1+\eta_{s,rel}) \gamma_{\delta}(\mu)/\alpha_{LB}(\mu).
    \label{eff_s_rel}
\end{equation}
\end{proposition}

\begin{proof}
Using {Eq.} \ref{error-eq1} with $v_{\delta} = \hat{r}_{\delta}(\mu)$ and the Cauchy-Schwarz inequality, it holds:
\begin{equation}
    \lVert \hat{r}_{\delta}(\mu) \rVert_{\mathbf{V}}^2 = a(e(\mu),\hat{r}_{\delta}(\mu);\mu) \le \lVert \hat{r}_{\delta}(\mu) \rVert_{\mu} \lVert e(\mu) \rVert_{\mu}.
    \label{r_hat_v}
\end{equation}
In addition, the continuity of the bilinear form $a(\cdot,\cdot;\mu)$ yields:
\begin{equation}
     \lVert \hat{r}_{\delta}(\mu) \rVert_{\mu}^2 =  a(\hat{r}_{\delta}(\mu) ,\hat{r}_{\delta}(\mu) ;\mu)\le \gamma_{\delta}(\mu)\lVert \hat{r}_{\delta}(\mu) \rVert_{\mathbf{V}}^2\le \gamma_{\delta}(\mu)\lVert \hat{r}_{\delta}(\mu) \rVert_{\mu}\lVert e(\mu)\rVert_{\mu}
     \label{r_hat_mu}
\end{equation}
Combining {Eqs.} \ref{r_hat_v} and \ref{r_hat_mu} we obtain: 
\begin{equation}
    \eta_{en}^2(\mu) = \frac{\lVert \hat{r}_{\delta}(\mu) \rVert_{\mathbf{V}}^2}{\alpha_{LB}(\mu)} \le \frac{\lVert \hat{r}_{\delta}(\mu) \rVert_{\mu}\lVert e(\mu) \rVert_{\mu}}{\alpha_{LB}(\mu)}\le \frac{\gamma_{\delta}(\mu)}{\alpha_{LB}(\mu)} \lVert e(\mu) \rVert_{\mu}^2,
\end{equation}
from which it follows {Eq.} \ref{eff_en}. \\ Using {Prop.} \ref{prop-err} yields:
\begin{equation}
    eff_s(\mu) = eff_{en}^2(\mu), 
\end{equation}
which proves {Eq.} \ref{eff_s}. \\ Finally, by definition of $\eta_{s,rel}$, it holds:
\begin{equation}
    eff_{s,rel}(\mu)=\frac{s_{\delta}(\mu)}{s_{rb}(\mu)}eff_s(\mu).
\end{equation}
Then, from {Prop.} \ref{prop-err} and {Eq.} \ref{seconda}, one obtains:
\begin{equation}
    \frac{s_{\delta}(\mu)}{s_{rb}(\mu)}=1+\frac{s_{\delta}(\mu)-s_{rb}(\mu)}{s_{rb}(\mu)} \le 1+ \frac{\eta_s(\mu)}{s_{\delta}(\mu)}=1+\eta_{s,rel}(\mu),
\end{equation}
which proves {Eq.} \ref{eff_s_rel}.
\end{proof}
So far, the error of the reduced solution $u_{\delta}(\mu)-u_{rb}(\mu)$ is controlled with the parameter dependent norm $\lVert \cdot \rVert_{\mu}$. 
However, it also proves useful to find an error estimation in a norm which is independent on $\mu$. So, for the fixed parameter $\bar{\mu}$, we refer to the norm over $\mathbf{V}$, $\lVert \cdot \rVert_{\mathbf{V}}$. \\ Let us define the error estimators:
\begin{equation}
    \eta_{\mathbf{V}}(\mu)=\frac{\lVert \hat{r}_{\delta}(\mu) \rVert_{\mathbf{V}}}{\alpha_{LB}(\mu)},
\end{equation}
\begin{equation}
        \eta_{\mathbf{V},rel}(\mu)=\frac{2 \lVert  \hat{r}_{\delta}(\mu) \rVert_{\mathbf{V}}}{\alpha_{LB}(\mu)\lVert u_{rb}(\mu) \rVert_{\mathbf{V}}}.
\end{equation}
\begin{proposition}
For all $\mu \in \mathbf{P}$
\begin{equation}
    \lVert u_{\delta}(\mu)-u_{rb}(\mu) \rVert_{\mathbf{V}} \le \eta_{\mathbf{V}}(\mu).
    \label{u_delta-u_rb}
\end{equation}
Moreover, if $\eta_{\mathbf{V},rel}(\mu) \le 1$:
\begin{equation}
    \frac{ \lVert u_{\delta}(\mu)-u_{rb}(\mu) \rVert_{\mathbf{V}}}{\lVert u_{\delta}(\mu)\rVert_{\mathbf{V}}}\le \eta_{\mathbf{V},rel}(\mu),
\end{equation}
for some $\mu \in \mathbf{P}$.
\end{proposition}

\begin{proof}
The first inequality can be proven similarly to what done for {Eq.} \ref{prima}. \\
For the second one, due to the hypothesis $\eta_{\mathbf{V},rel}(\mu) \le 1$, it holds:
\begin{equation}
\begin{split}
    \lVert u_{\delta}(\mu) \rVert_{\mathbf{V}} & = \lVert u_{rb}(\mu) \rVert_{\mathbf{V}}+\lVert u_{\delta}(\mu) \rVert_{\mathbf{V}} -\lVert u_{rb}(\mu) \rVert_{\mathbf{V}}  \ge \lVert u_{rb}(\mu) \rVert_{\mathbf{V}} - \lVert u_{\delta}(\mu) - u_{rb}(\mu) \rVert_{\mathbf{V}}
    \\
    & \ge \lVert u_{rb}(\mu) \rVert_{\mathbf{V}} -\eta_{\mathbf{V}}(\mu) = \big(1-\frac{1}{2}\eta_{\mathbf{V},rel}(\mu)\big)\lVert u_{rb}(\mu) \rVert_{\mathbf{V}} \ge \frac{1}{2}\lVert u_{rb}(\mu) \rVert_{\mathbf{V}}.
\end{split}
\end{equation}
Finally, it yields:
\begin{equation*}
    \eta_{\mathbf{V},rel}(\mu)=2\frac{ \lVert  \hat{r}_{\delta}(\mu) \rVert_{\mathbf{V}}}{\alpha_{LB}(\mu)\lVert u_{rb}(\mu) \rVert_{\mathbf{V}}}=2\frac{\lVert u_{\delta}(\mu) \rVert_{\mathbf{V}}}{\lVert u_{rb}(\mu) \rVert_{\mathbf{V}}} 
    \frac{\eta_{\mathbf{V}}(\mu)}{\lVert u_{\delta}(\mu) \rVert_{\mathbf{V}}} \ge \frac{\eta_{\mathbf{V}}(\mu)}{\lVert u_{\delta}(\mu) \rVert_{\mathbf{V}}} \ge \frac{\lVert u_{\delta}(\mu)-u_{rb}(\mu) \rVert_{\mathbf{V}}}{\lVert u_{\delta}(\mu) \rVert_{\mathbf{V}}},
\end{equation*}
which concludes the proof.
\end{proof}
The associated effectivities indexes can be defined as:
\begin{equation}
    eff_{\mathbf{V}}(\mu) = \frac{\eta_{\mathbf{V}}(\mu)}{\lVert u_{\delta}(\mu)-u_{rb}(\mu) \rVert_{\mathbf{V}}}
\end{equation}
\begin{equation}
        eff_{\mathbf{V},rel}(\mu) = \frac{\eta_{\mathbf{V},rel}(\mu)\lVert u_{\delta}(\mu) \rVert_{\mathbf{V}}}{\lVert u_{\delta}(\mu)-u_{rb}(\mu) \rVert_{\mathbf{V}}}
\end{equation}

\begin{proposition}
For all $\mu \in \mathbf{P}$
\begin{equation}
    eff_{\mathbf{V}}(\mu)\le \frac{\gamma_{\delta}(\mu)}{\alpha_{LB}(\mu)}.
\end{equation}
Moreover, if $\eta_{\mathbf{V},rel}(\mu)\le 1$:
\begin{equation}
        eff_{\mathbf{V},rel}(\mu)\le 3 \frac{\gamma_{\delta}(\mu)}{\alpha_{LB}(\mu)},
\end{equation}
for some $\mu \in \mathbf{P}$.
\end{proposition}
\begin{proof}
Using the inequality $\lVert e(\mu) \rVert_{\mu} \le \gamma_{\delta}(\mu)\lVert e(\mu) \rVert_{\mathbf{V}}$ and {Eq.} \ref{eff_en}, it holds the first statement:
\begin{equation}
    eff_{\mathbf{V}}(\mu) = \frac{eff_{en}(\mu)}{\alpha_{LB}(\mu)}\le\frac{\gamma_{\delta}^{1/2}(\mu)}{\alpha_{LB}(\mu)}\frac{\lVert e(\mu) \rVert_{\mu}}{\lVert e(\mu) \rVert_{\mathbf{V}}}\le \frac{\gamma_{\delta}(\mu)}{\alpha_{LB}(\mu)}.
\end{equation}
As for the second inequality, combining the hypothesis $\eta_{\mathbf{V},rel}(\mu)\le 1$ and {Eq.} \ref{u_delta-u_rb}, one obtains:
\begin{equation}
\begin{split}
    \lVert u_{\delta}(\mu) \rVert_{\mathbf{V}} - & \lVert u_{rb}(\mu) \rVert_{\mathbf{V}}\le \lVert u_{\delta}(\mu)-u_{rb}(\mu) \rVert_{\mathbf{V}}
    \\
    & \le \frac{\lVert \hat{r}_{\delta}(\mu) \rVert_{\mathbf{V}}}{\alpha_{LB}(\mu)} = \frac{1}{2} \lVert u_{rb}(\mu) \rVert_{\mathbf{V}} \eta_{\mathbf{V},rel}(\mu)\le \frac{1}{2} \lVert u_{rb}(\mu) \rVert_{\mathbf{V}} .
\end{split}
\end{equation}
Finally, we get
\begin{equation*}
    eff_{\mathbf{V},rel}(\mu) = 2 \frac{\lVert u_{\delta}(\mu) \rVert_{\mathbf{V}}}{\lVert u_{rb}(\mu) \rVert_{\mathbf{V}}} eff_{\mathbf{V}}(\mu) = 2 \Bigg(1+ \frac{\lVert u_{\delta}(\mu) \rVert_{\mathbf{V}}-\lVert u_{rb}(\mu) \rVert_{\mathbf{V}}}{\lVert u_{rb}(\mu) \rVert_{\mathbf{V}}} \Bigg)eff_{\mathbf{V}}(\mu) \le 3 \frac{\gamma_{\delta}(\mu)}{\alpha_{LB}(\mu)}.
\end{equation*}

\end{proof}

\section{CONCLUSIONS}\label{concl}
The aim of this chapter is to provide to the reader a brief overview about the ROM framework. 
However, the main ingredients here presented can be useful to study many applications and sometimes they can be easily extended to more complex cases, as reported in others chapters of the book. 

In this context, we considered an affine linear coercive formulation in a finite element environment for a simple explanation. We investigated the POD and the greedy algorithm in order to understand how to extract the reduced space. On the other hand, the Galerkin projection is applied to compute the coefficients of the reduced solution. Moreover, we
gave some insights related to the posterior error estimation. 

To conclude, it is clear that the \emph{offline}-\emph{online} paradigm enables to improve the computational speed-up in many scenarios because the dimension of the reduced space is typically much smaller than the full order one. 

\section*{ACKNOWLEDGMENTS}
We acknowledge the support provided by the European Research Council Executive Agency by the Consolidator Grant project AROMA-CFD "Advanced Reduced Order Methods with Applications in Computational Fluid Dynamics" - GA 681447, H2020-ERC CoG 2015 AROMA-CFD, PI G. Rozza, and INdAM-GNCS 2019-2020 projects. 
\bibliographystyle{elsarticle-harv}  
\bibliography{references.bib}



\end{document}